\documentclass[oneside]{amsart}
\usepackage[utf8]{inputenc}
\usepackage{mathrsfs}
\usepackage{bm}
\usepackage{amsthm}
\usepackage{amstext}
\usepackage{amssymb}

\makeatletter
\numberwithin{equation}{section}
\numberwithin{figure}{section}
  \theoremstyle{remark}
  \newtheorem*{acknowledgement*}{\protect\acknowledgementname}
\theoremstyle{plain}
\newtheorem{thm}{\protect\theoremname}[section]
  \theoremstyle{definition}
  \newtheorem{defn}[thm]{\protect\definitionname}
  \theoremstyle{remark}
  \newtheorem{rem}[thm]{\protect\remarkname}
  \theoremstyle{plain}
  \newtheorem{prop}[thm]{\protect\propositionname}
  \theoremstyle{plain}
  \newtheorem{lem}[thm]{\protect\lemmaname}
  \theoremstyle{plain}
  \newtheorem{cor}[thm]{\protect\corollaryname}

\usepackage{a4wide}
\usepackage{url}

\usepackage{xcolor}
\usepackage{version}

\excludeversion{NB}

\newcommand{\Hom}{\operatorname{Hom}}

\newcommand{\Uq}{\mathbf{U}_q}

\newcommand{\wt}{\operatorname{wt}}
\newcommand{\height}{\operatorname{ht}}

\newcommand{\up}{\mathrm{up}}
\newcommand{\low}{\mathrm{low}}
\newcommand{\Glow}{G^{\low}}
\newcommand{\Gup}{G^{\up}}
\newcommand{\dprod}{\mathop{\overrightarrow{\prod}}\limits}

\newcommand{\Ker}{\operatorname{Ker}}

\makeatother

  \providecommand{\acknowledgementname}{Acknowledgement}
  \providecommand{\corollaryname}{Corollary}
  \providecommand{\definitionname}{Definition}
  \providecommand{\lemmaname}{Lemma}
  \providecommand{\propositionname}{Proposition}
  \providecommand{\remarkname}{Remark}
\providecommand{\theoremname}{Theorem}

\begin{document}

\title{Quantum twist maps and dual canonical bases}

\author{Yoshiyuki Kimura}

\address{Yoshiyuki Kimura, Department of Mathematics, Graduate School of Science,
Kobe University, 1-1, Rokkodai, Nada-ku, Kobe 657-8501, Japan}

\email{ykimura@math.kobe-u.ac.jp}

\author{Hironori Oya}

\address{Hironori Oya, Graduate School of Mathematical Sciences, The University
of Tokyo, Komaba, Tokyo, 153-8914, Japan}

\email{oya@ms.u-tokyo.ac.jp}

\thanks{The work of the first author was supported by JSPS Grant-in-Aid for Scientific
Research (S) 24224001. The work of the second author was supported by
Grant-in-Aid for JSPS Fellows (No.~15J09231) and the Program for
Leading Graduate Schools, MEXT, Japan.}
\begin{abstract}
In this paper, we show that quantum twist maps, introduced by Lenagan-Yakimov, induce bijections between dual canonical bases of quantum nilpotent subalgebras. As a corollary, we show the unitriangular property between dual canonical bases and Poincar\'e-Birkhoff-Witt type bases under the ``reverse'' lexicographic order. We also show that quantum twist maps induce bijections between certain unipotent quantum minors.
\end{abstract}

\maketitle

\section{Introduction}
\label{intro}
Let $\mathfrak{g}$ be a symmetrizable Kac-Moody Lie algebra and $\Uq:=\Uq(\mathfrak{g})$ the corresponding quantized universal enveloping algebra. We mainly focus on the subalgebra $\Uq^-(w)$, called the quantum nilpotent subalgebra, of the negative half $\Uq^-$ of $\Uq$, which is determined by an element $w$ of the Weyl group $W$ of $\mathfrak{g}$. The quantum nilpotent subalgebra $\Uq^-(w)$ is a quantum analogue of the coordinate ring of the unipotent group associated with $w$, and it has some nice structural properties. For example, it is known that $\Uq^-(w)$ has (dual) Poincar\'e-Birkhoff-Witt type bases and the dual canonical basis. Here the dual canonical basis is given by a subset of the dual canonical basis (=the upper global basis) of $\Uq^-$ defined by Lusztig \cite{MR1035415} and subsequently by Kashiwara \cite{MR1115118}. Moreover, there exists a quantum cluster algebra structure on $\Uq^-(w)$ \cite{GLS:qcluster,goodearlyakimov}. 

Lenagan-Yakimov \cite[Section 6]{lenagan2015prime} and Goodearl-Yakimov \cite[Section 8]{goodearl2016berenstein} introduced quantum analogues $\Theta_w$ ($w\in W$) of the Fomin-Zelevinsky twist maps  \cite{MR1652878,MR3096792} for quantized universal enveloping algebras. The quantum twist map $\Theta_w$ induces the anti-algebra isomorphism from $\Uq^-(w^{-1})$ to $\Uq^-(w)$. In this paper, we show that the quantum twist maps are restricted to the bijections between the dual canonical bases of quantum nilpotent subalgebras (Theorem \ref{t:mainthm}). As a corollary, we show the unitriangular property between dual canonical bases and dual Poincar\'e-Birkhoff-Witt type bases under the ``reverse'' lexicographic order (Corollary \ref{c:right-lex}). This is new when $\mathfrak{g}$ is not of finite type. (We discuss finite types in Section 4. See Remark \ref{r:finite}.) Moreover there are the specific dual canonical basis elements, called unipotent quantum minors. We also show that quantum twist maps induce bijections between certain unipotent quantum minors (Theorem \ref{t:mainthm2}). This result is a quantum analogue of \cite[Lemma 2.25]{MR1652878}. In particular, quantum twist maps preserve quantum $T$-systems (Corollary \ref{c:T-sys}). We should also remark that our results include the slight refinement of \cite[Proposition 6.1]{lenagan2015prime}(cf.~\cite[The equality (6.7)]{lenagan2015prime}).

The following are some related topics and the explicit relations between our results and them should be explored in future research: \hspace{-10pt}
\begin{itemize}
\item In the case of finite type, under the categorification via the quiver Hecke algebras, the ``reverse'' unitriangular property for (proper) standard modules is also proved by McNamara \cite[Theorem 3.1 (5)]{MR3403455}. 
\item There exists a variant of the twist maps treated in \cite{MR1405449,MR1456321,MR2833478} and its quantum analogue is proposed in \cite[Theorem 2.10, Conjecture 2.12 (c)]{MR3397447}. Berenstein-Rupel conjectured that it also preserves the dual canonical basis \cite[Conjecture 2.17 (a)]{MR3397447}. 
\end{itemize}

\begin{acknowledgement*}
The authors would like to express our sincere gratitude to Yoshihisa Saito, the supervisor of the second author, for his helpful comments.
\end{acknowledgement*}

\section{Preliminaries}\label{Preliminaries}

\subsection{Quantized universal enveloping algebras}\label{Quantized universal enveloping algebras}

Throughout this paper $\mathfrak{g}$ is supposed to be a symmetrizable
Kac-Moody Lie algebra over $\mathbb{C}$ unless otherwise specified.
We follow the notation in \cite{MR1357199}. However, to shorten notation,
we write $q^{h}$ instead of $q(h)$.
\begin{defn}
Let $\left(P,I,\left\{ \alpha_{i}\right\} _{i\in I},\left\{ h_{i}\right\} _{i\in I},\left(\;,\;\right)\right)$
be a root datum for $\mathfrak{g}$ and $q$ an indeterminate. Denote
by $\Uq:=\Uq(\mathfrak{g})$ the quantized universal enveloping algebra over $\mathbb{Q}(q)$ generated by $\{e_{i},f_{i},q^{h}\}_{i\in I,h\in P^{*}}$. Here $P^{*}:=\Hom_{\mathbb{Z}}(P, \mathbb{Z})$. Set the coproduct $\Delta\colon\Uq\to\Uq\otimes\Uq$, the antipode $S\colon\Uq\to\Uq$ and the counit $\varepsilon\colon\Uq\to\mathbb{Q}(q)$
as follows:
\begin{align*}
\Delta\left(e_{i}\right) & =e_{i}\otimes t_{i}^{-1}+1\otimes e_{i}, & S\left(e_{i}\right) & =-e_{i}t_{i}, & \varepsilon\left(e_{i}\right) & =0,\\
\Delta\left(f_{i}\right) & =f_{i}\otimes1+t_{i}\otimes f_{i}, & S\left(f_{i}\right) & =-t_{i}^{-1}f_{i}, & \varepsilon\left(f_{i}\right) & =0,\\
\Delta\left(q^{h}\right) & =q^{h}\otimes q^{h}, & S\left(q^{h}\right) & =q^{-h}, & \varepsilon\left(q^{h}\right) & =1.
\end{align*}
We denote by $\Uq^{+},\Uq^{-},\Uq^{0}$
the $\mathbb{Q}(q)$-subalgebras of $\Uq$ generated by
$\{e_{i}\}_{i\in I}$, $\{f_{i}\}_{i\in I}$, $\{q^{h}\}_{h\in P^{*}}$ respectively. We assume that there exist elements $\{\varpi_i\}_{i\in I}$ of $P$ such that $\langle \varpi_i, h_j\rangle=\delta_{ij}$ for all $i, j\in I$. Set $\rho:=\sum_{i\in I}\varpi_i$. Then $\langle\rho, h_{i}\rangle=1$ and $(\rho,\alpha_{i})=(\alpha_{i},\alpha_{i})/2$ for all $i\in I$.
\end{defn}

\begin{defn}
Let $\vee\colon\Uq\to\Uq$ be the $\mathbb{Q}(q)$-algebra
involution defined by
\begin{align*}
e_{i}^{\vee} & =f_{i}, &  & f_{i}^{\vee}=e_{i}, &  & \left(q^{h}\right)^{\vee}=q^{-h}.
\end{align*}
Let $\overline{\phantom{x}}\colon\mathbb{Q}(q)\to\mathbb{Q}(q)$, $\overline{\phantom{x}}\colon\Uq\to\Uq$
be the $\mathbb{Q}$-algebra involutions defined by
\begin{align*}
\overline{q}=q^{-1}, &  & \overline{e_{i}}=e_{i}, &  & \overline{f_{i}}=f_{i}, &  & \overline{q^{h}} & =q^{-h}.
\end{align*}
Let $*\colon\Uq\to\Uq$ , $\varphi\colon\Uq\to\Uq$
be the $\mathbb{Q}\left(q\right)$-algebra anti-involutions defined
by
\begin{align*}
*(e_{i}) & =e_{i}, & *(f_{i}) & =f_{i}, & *\left(q^{h}\right) & =q^{-h},\\
\varphi\left(e_{i}\right) & =f_{i}, & \varphi\left(f_{i}\right) & =e_{i}, & \varphi\left(q^{h}\right) & =q^{h}.
\end{align*}

\end{defn}

\subsection{Non-degenerate pairings and dual bar-involutions}\label{Non-degenerate pairings and dual bar-involutions}

\begin{defn}\label{d:qderiv}
For $i\in I$, we define the $\mathbb{Q}(q)$-linear maps $e'_{i}$
and $_{i}e'\colon\Uq^{-}\to\Uq^{-}$ by
\[
\begin{array}{lc}
e'_{i}\left(xy\right)=e'_{i}\left(x\right)y+q_{i}^{\langle\wt x,h_{i}\rangle}xe'_{i}\left(y\right), & e'_{i}(f_{j})=\delta_{ij},\\
_{i}e'\left(xy\right)=q_{i}^{\langle\wt y,h_{i}\rangle}{_{i}e'}\left(x\right)y+x\;{_{i}e'}\left(y\right), & _{i}e'(f_{j})=\delta_{ij}
\end{array}
\]
for homogeneous elements $x,y\in\Uq^{-}$. Here homogeneous
elements mean the elements $x$ of $\Uq$ satisfying $q^{h}xq^{-h}=q^{\langle\wt x,h\rangle}x$
for some $\wt x\in Q:=\sum_{i\in I}\mathbb{Z}\alpha_i$ and an arbitrary $h\in P^{*}$.
\end{defn}

\begin{defn}
There exists a unique symmetric $\mathbb{Q}(q)$-bilinear form $(\ ,\ )_{L}\colon\Uq^{-}\times\Uq^{-}\to\mathbb{Q}(q)$
such that
\begin{align*}
(1,1)_{L}=1 & , &  & (f_{i}x,y)_{L}=\frac{1}{1-q_{i}^{2}}(x,e'_{i}(y))_{L}, &  & (xf_{i},y)_{L}=\frac{1}{1-q_{i}^{2}}(x,{_{i}e'}(y))_{L}.
\end{align*}
This form $(\ ,\ )_{L}$ is non-degenerate and has the following property:
\[
\left(\ast(x),\ast(y)\right)_{L}=\left(x,y\right)_{L}
\]
for all $x,y\in\Uq^{-}$. See \cite[Chapter 1]{Lus:intro} for more
details.
\end{defn}

\begin{defn}
For a homogeneous $x\in\Uq^{-}$, we define $\sigma\left(x\right)=\sigma_{L}\left(x\right)\in\Uq^{-}$
by the following property: 
\[
\left(\sigma\left(x\right),y\right)_{L}=\overline{\left(x,\overline{y}\right)}_{L}
\]
for all $y\in\Uq^{-}$. By non-degeneracy of $(\;,\;)_{L}$,
the element $\sigma\left(x\right)$ is well-defined. This map $\sigma\colon\Uq^{-}\to\Uq^{-}$
is called the dual bar-involution.
\end{defn}
The following proposition can be proved in the same manner as \cite[Proposition 3.2]{MR2914878}.
\begin{prop}
\label{p:dualbar} For a homogeneous element $x\in\Uq^{-}$,
we have
\[
\sigma\left(x\right)=\left(-1\right)^{\height\left(\wt x\right)}q^{\left(\wt x,\wt x\right)/2-\left(\wt x,\rho\right)}\left(\overline{\phantom{x}}\circ*\right)\left(x\right).
\]
Here $\height\alpha:=\sum_{i\in I}m_{i}$ for $\alpha=\sum_{i\in I}m_{i}\alpha_{i}\in Q$.
In particular, for homogeneous elements $x,y\in\Uq^{-}$,
we have
\[
\sigma(xy)=q^{(\wt x,\wt y)}\sigma(y)\sigma(x).
\]

\end{prop}

\subsection{Lusztig's braid group symmetries}\label{Lusztig's braid group symmetries}

\begin{defn}
Let $W$ be the Weyl group of $\mathfrak{g}$ and $\left\{ s_{i}\right\} _{i\in I}$
be the set of simple reflections. For $w\in W$, denote by $I(w)$
the set of the reduced words of $w$.

Following Lusztig \cite[Section 37.1.3]{Lus:intro}, we define the
$\mathbb{Q}\left(q\right)$-algebra automorphism $T_{i}\colon\Uq\to\Uq$
for $i\in I$ by the following formulae: \begin{subequations}
\begin{align*}
T_{i}\left(q^{h}\right) & =q^{s_{i}\left(h\right)},\\
T_{i}\left(e_{j}\right) & =\begin{cases}
-f_{i}t_{i} & \text{for}\;j=i,\\
{\displaystyle \sum_{r+s=-\left\langle h_{i},\alpha_{j}\right\rangle }\left(-1\right)^{r}q_{i}^{-r}e_{i}^{\left(s\right)}e_{j}e_{i}^{\left(r\right)}} & \text{for}\;j\neq i,
\end{cases}\\
T_{i}\left(f_{j}\right) & =\begin{cases}
-t_{i}^{-1}e_{i} & \text{for}\;j=i,\\
{\displaystyle \sum_{r+s=-\left\langle h_{i},\alpha_{j}\right\rangle }\left(-1\right)^{r}q_{i}^{r}f_{i}^{\left(r\right)}f_{j}f_{i}^{\left(s\right)}} & \text{for}\;j\neq i.
\end{cases}
\end{align*}
\end{subequations} Its inverse map is given by \begin{subequations}
\begin{align*}
T_{i}^{-1}\left(q^{h}\right) & =q^{s_{i}\left(h\right)},\\
T_{i}^{-1}\left(e_{j}\right) & =\begin{cases}
-t_{i}^{-1}f_{i} & \text{for}\;j=i,\\
{\displaystyle \sum_{r+s=-\left\langle h_{i},\alpha_{j}\right\rangle }\left(-1\right)^{r}q_{i}^{-r}e_{i}^{\left(r\right)}e_{j}e_{i}^{\left(s\right)}} & \text{for}\;j\neq i,
\end{cases}\\
T_{i}^{-1}\left(f_{j}\right) & =\begin{cases}
-e_{i}t_{i} & \text{for}\;j=i,\\
{\displaystyle \sum_{r+s=-\left\langle h_{i},\alpha_{j}\right\rangle }\left(-1\right)^{r}q_{i}^{r}f_{i}^{\left(s\right)}f_{j}f_{i}^{\left(r\right)}} & \text{for}\;j\neq i.
\end{cases}
\end{align*}
\end{subequations} The maps $T_{i}$ and $T_{i}^{-1}$ are denoted
by $T''_{i,1}$ and $T'_{i,-1}$ respectively in \cite{Lus:intro}.

It is known that $\left\{ T_{i}\right\} _{i\in I}$ satisfies the
braid relations, that is, for $w\in W$, the $\mathbb{Q}\left(q\right)$-algebra
automorphism $T_{w}:=T_{i_{1}}\cdots T_{i_{\ell}}\colon\Uq\to\Uq$ does
not depend on the choice of $(i_{1},\dots,i_{\ell})\in I(w)$. See \cite[Chapter 39]{Lus:intro}.
\end{defn}

The following lemma follows from the straightforward check on the
generators of $\Uq$.
\begin{lem}
\label{l:T-antipode} For $i\in I$, we have $T_{i}\circ S\circ\vee=S\circ\vee\circ T_{i}^{-1}$.
\end{lem}

We have the following invariance of the bilinear form $(\ ,\ )_{L}$
under the braid group symmetry $T_{i}$.
\begin{prop}[{{{\cite[Proposition 38.1.6, Proposition 38.2.1]{Lus:intro}}}}]
\label{p:compatibility}

\textup{(1)} For $i\in I$, we have $\Ker e'_{i}=\Uq^{-}\cap T_{i}\Uq^{-}.$

\textup{(2)} For $i\in I$ and $x,y\in\Ker e'_{i}$, we have $\left(x,y\right)_{L}=\left(T_{i}^{-1}(x),T_{i}^{-1}(y)\right)_{L}$.
\end{prop}

\subsection{Quantum nilpotent subalgebras and Poincar\'e-Birkhoff-Witt type bases}\label{Quantum nilpotent subalgebras and Poincare-Birkhoff-Witt bases}
\begin{defn}
(1) For $w\in W$, we set $\Uq^{-}\left(w\right)=\Uq^{-}\cap T_{w}\left(\Uq^{+}\Uq^{0}\right)$. These subalgebras of $\Uq^-$ are called quantum nilpotent subalgebras.

(2) Let $w\in W$ and $\bm{i}=\left(i_{1},\cdots,i_{\ell}\right)\in I\left(w\right)$.
For $\bm{c}=\left(c_{1},\cdots,c_{\ell}\right)\in\mathbb{Z}_{\geq0}^{\ell}$,
we set
\begin{align*}
F^{{\rm low}}\left(\bm{c},\bm{i}\right) & :=f_{i_{1}}^{\left(c_{1}\right)}T_{i_{1}}\left(f_{i_{2}}^{\left(c_{2}\right)}\right)\cdots\left(T_{i_{1}}\cdots T_{i_{\ell-1}}\right)\left(f_{i_{\ell}}^{\left(c_{\ell}\right)}\right),\\
F^{\mathrm{up}}\left(\bm{c},\bm{i}\right) & :=F^{{\rm low}}\left(\bm{c},\bm{i}\right)/\left(F^{{\rm low}}\left(\bm{c},\bm{i}\right),F^{{\rm low}}\left(\bm{c},\bm{i}\right)\right)_{L}.
\end{align*}
\end{defn}
\begin{prop}[{\cite[Proposition 38.2.3]{Lus:intro}, \cite[Proposition 2.3]{MR1712630}}]
\textup{(1)} $F^{{\rm low}}\left(\bm{c},\bm{i}\right)\in\Uq^{-}\left(w\right)$
for $\bm{c}\in\mathbb{Z}_{\geq0}^{\ell}$ and $\{F^{{\rm low}}\left(\bm{c},\bm{i}\right)\}_{\bm{c}\in\mathbb{Z}_{\geq0}^{\ell}}$
forms a basis of $\Uq^{-}\left(w\right)$.

\textup{(2)} $\{F^{{\rm low}}\left(\bm{c},\bm{i}\right)\}_{\bm{c}\in\mathbb{Z}_{\geq0}^{\ell}}$
is an orthogonal basis of $\Uq^{-}\left(w\right)$, more precisely, we have
\begin{align}
(F^{{\rm low}}\left(\bm{c},\bm{i}\right),F^{{\rm low}}\left(\bm{c}',\bm{i}\right))_{L}=\delta_{\bm{c},\bm{c}'}\prod_{k=1}^{\ell}\prod_{j=1}^{c_{k}}(1-q_{i_{k}}^{2j})^{-1} & .\label{norm}
\end{align}

\end{prop}
In particular, $\{F^{{\rm up}}\left(\bm{c},\bm{i}\right)\}_{\bm{c}\in\mathbb{Z}_{\geq0}^{\ell}}$
is also a basis of $\Uq^{-}\left(w\right)$. The basis $\{F^{{\rm low}}\left(\bm{c},\bm{i}\right)\}_{\bm{c}}$ is called the (lower) Poincar\'e-Birkhoff-Witt type basis associated with $\bm{i}\in I\left(w\right)$, and the basis $\{F^{{\rm up}}\left(\bm{c},\bm{i}\right)\}_{\bm{c}}$ is called the dual (or upper) Poincar\'e-Birkhoff-Witt type basis.

\subsection{Dual canonical bases and quantum nilpotent subalgebras}\label{Dual canonical bases and quantum nilpotent subalgebras}
Set $\mathcal{A}:=\mathbb{Z}[q^{\pm1}]$. Denote by $_{\mathcal{A}}\Uq^{-}$
the $\mathcal{A}$-subalgebra of $\Uq^{-}$ generated by the elements
$\{f_{i}^{(n)}\}_{i\in I,n\in\mathbb{Z}_{\geq0}}$. Lusztig \cite{MR1035415,Lus:intro}
and Kashiwara \cite{MR1115118} have proved that $_{\mathcal{A}}\Uq^{-}$
is a free $\mathcal{A}$-module and constructed the basis $\mathbf{B}^{\mathrm{low}}$
of $\Uq^{-}$, called the canonical basis, which is also an $\mathcal{A}$-basis
of $_{\mathcal{A}}\Uq^{-}$. Moreover the elements of $\mathbf{B}^{\mathrm{low}}$
are parametrized by the Kashiwara crystal $\mathscr{B}(\infty)$.
We follow the notation in \cite{MR1357199} concerning the crystal $(\mathscr{B}(\infty); \wt, \{\tilde{e}_i\}_{i\in I}, \{\tilde{f}_i\}_{i\in I}, \{\varepsilon_i\}_{i\in I}, \{\varphi_i\}_{i\in I})$ and the $\ast$-crystal $(\mathscr{B}(\infty); \wt, \{\tilde{e}_i^{\ast}\}_{i\in I}, \{\tilde{f}_i^{\ast}\}_{i\in I}, \{\varepsilon_i^{\ast}\}_{i\in I}, \{\varphi_i^{\ast}\}_{i\in I})$. Write $\mathbf{B}^{\mathrm{low}}=\{\Glow(b)\}_{b\in\mathscr{B}(\infty)}$.
We have $\overline{\Glow(b)}=\Glow(b)$ for all $b\in\mathscr{B}(\infty)$ \cite[Lemma 7.3.4]{MR1115118}.

Denote by $\mathbf{B}^{\mathrm{up}}$ the basis of $\Uq^{-}$ dual
to $\mathbf{B}^{\mathrm{low}}$ with respect to the bilinear form
$(\ ,\ )_{L}$, that is, $\mathbf{B}^{\mathrm{up}}=\{\Gup(b)\}_{b\in\mathscr{B}(\infty)}$
such that
\[
(\Glow(b),\Gup(b'))_{L}=\delta_{b,b'}
\]
for any $b,b'\in B(\infty)$. Hence we have $\sigma(\Gup(b))=\Gup(b)$
for all $b\in\mathscr{B}(\infty)$. It is known that the dual canonical basis $\mathbf{B}^{\mathrm{up}}$ is compatible with the quantum nilpotent subalgebras as follows. 
\begin{prop}[{{\cite[Theorem 4.25, Theorem 4.29]{MR2914878}}}]
\label{p:dualcanonical} Let $w\in W$ and $\bm{i}\in I\left(w\right)$.

\textup{\textup{(1)}} $\Uq^{-}(w)\cap\mathbf{B}^{\mathrm{up}}$ is
a basis of $\Uq^{-}(w)$.

\textup{\textup{(2)}} each element $\Gup(b)$ of $\Uq^{-}(w)\cap\mathbf{B}^{\mathrm{up}}$ is characterized by the following conditions:
\begin{enumerate}
\item[(DCB1)] $\sigma(\Gup(b))=\Gup(b)$, and
\item[(DCB2)] $\Gup\left(b\right)=F^{{\rm {up}}}\left(\bm{c},\bm{i}\right)+\sum_{\bm{c}'<\bm{c}}d_{\bm{c},\bm{c}'}^{\bm{i}}F^{{\rm {up}}}\left(\bm{c}',\bm{i}\right)$
with $d_{\bm{c},\bm{c}'}^{\bm{i}}\in q\mathbb{Z}[q]$ for a unique
$\bm{c}\in\mathbb{Z}_{\geq0}^{\ell}$.
\end{enumerate}
Here $<$ denotes the left lexicographic order on $\mathbb{Z}_{\geq0}^{\ell}$, that is, we write $(c_1,\dots, c_{\ell})<(c'_1,\dots, c'_{\ell})$ if and only if there exists $k\in \{1,\dots, \ell\}$ such that $c_1=c'_1,\dots, c_{k-1}=c'_{k-1}$ and $c_k<c'_k$.
\end{prop}
\begin{defn}\label{d:PBWparam}
Proposition \ref{p:dualcanonical} (2) says that $F^{{\rm {up}}}\left(\bm{c},\bm{i}\right)$ determines a unique dual canonical basis element $\Gup(b)$ in $\Uq^{-}(w)$. We write the corresponding element of $\mathscr{B}(\infty)$ as $b\left(\bm{c},\bm{i}\right)$. Then $\Uq^{-}(w)\cap\mathbf{B}^{\mathrm{up}}=\{\Gup(b(\bm{c},\bm{i}))\}_{\bm{c}\in\mathbb{Z}_{\geq0}^{\ell}}$.
\end{defn}
\begin{rem}
\label{r:dualcanonical} The unitriangular property (DCB2) in Proposition
\ref{p:dualcanonical} is equivalent to the following unitriangular
property:
\begin{align*}
F^{\mathrm{up}}\left(\bm{c},\bm{i}\right)=\sum\nolimits _{\bm{c}'\in\mathbb{Z}_{\geq0}^{\ell}}\left[F^{\mathrm{up}}\left(\bm{c},\bm{i}\right)\colon G^{{\rm up}}\left(b\left(\bm{c}',\bm{i}\right)\right)\right]G^{{\rm up}}\left(b\left(\bm{c}',\bm{i}\right)\right)\;\text{with}\\
\;\left[F^{\mathrm{up}}\left(\bm{c},\bm{i}\right)\colon G^{{\rm up}}\left(b\left(\bm{c}',\bm{i}\right)\right)\right]\begin{cases}
\in\delta_{\bm{c}',\bm{c}}+(1-\delta_{\bm{c}',\bm{c}})q\mathbb{Z}[q] & \text{if}\;\bm{c}'\leq\bm{c}\\
=0 & \text{otherwise.}
\end{cases}
\end{align*}
In fact, these unitriangular properties also hold when we consider
the right lexicographic order on $\mathbb{Z}_{\geq0}^{\ell}$. See
Corollary \ref{c:right-lex} below.
\end{rem}
\begin{rem}
\label{r:characterize} Note that the element $\Gup\left(b\left(\bm{c},\bm{i}\right)\right)$
is already characterized by the property (DCB1) in Proposition \ref{p:dualcanonical}
and the following property:
\begin{enumerate}
\item[(DCB2)$'$] $\Gup\left(b\left(\bm{c},\bm{i}\right)\right)-F^{{\rm {up}}}\left(\bm{c},\bm{i}\right)\in\sum_{\bm{c}'\in\mathbb{Z}_{\geq0}^{\ell}}q\mathbb{Z}[q]F^{{\rm {up}}}\left(\bm{c}',\bm{i}\right)$.
\end{enumerate}
Indeed, if an element $x\in \Uq^{-}(w)$ satisfies (DCB2)$'$, that is, $x-F^{{\rm {up}}}\left(\bm{c},\bm{i}\right)\in\sum_{\bm{c}'\in\mathbb{Z}_{\geq0}^{\ell}}q\mathbb{Z}[q]F^{{\rm {up}}}\left(\bm{c}',\bm{i}\right)$, we have 
\[
x\in \Gup\left(b\left(\bm{c},\bm{i}\right)\right)+\sum\nolimits_{\bm{c}'\in\mathbb{Z}_{\geq0}^{\ell}}q\mathbb{Z}[q]\Gup\left(b\left(\bm{c}',\bm{i}\right)\right)
\]
by Remark \ref{r:dualcanonical}. Hence $x$ is equal to $\Gup\left(b\left(\bm{c},\bm{i}\right)\right)$ if $x$ also satisfies (DCB1) because it is the unique $\sigma$-invariant element in $\Gup\left(b\left(\bm{c},\bm{i}\right)\right)+\sum\nolimits_{\bm{c}'\in\mathbb{Z}_{\geq0}^{\ell}}q\mathbb{Z}[q]\Gup\left(b\left(\bm{c}',\bm{i}\right)\right)$. 
\end{rem}
\begin{prop}[{\cite[Proposition 4.26]{MR2914878}, \cite[Proposition 7.4]{GLS:qcluster}}]
\label{p:rootdual} For $k=1,\dots,\ell$, we set $\bm{c}_{k}:=\left(\delta_{jk}\right)_{1\leq j\leq \ell}\in \mathbb{Z}_{\geq0}^{\ell}$. 
Then we have $\Gup(b(\bm{c}_{k},\bm{i}))=F^{{\rm {up}}}\left(\bm{c}_{k},\bm{i}\right)$.
\end{prop}

\section{Quantum twist maps}\label{Quantum twist maps}

In this section, we study the compatibility between quantum twist maps and dual canonical bases. We consider the quantum twist maps defined in \cite[Section 6]{lenagan2015prime}.

\subsection{Quantum twist maps and quantum nilpotent subalgebras}\label{Quantum twist maps and quantum nilpotent subalgebras}
\begin{defn}
For $w\in W$, we consider the $\mathbb{Q}(q)$-algebra anti-automorphism
$\Theta_{w}$ of $\Uq$ defined by
\begin{align*}
\Theta_{w} & :=T_{w}\circ S\circ\vee.
\end{align*}
By Lemma \ref{l:T-antipode} and $(S\circ\vee)^{2}=\mathrm{id}$,
we have $(\Theta_{w})^{-1}=\Theta_{w^{-1}}$. For a homogeneous element
$x\in\Uq$, we have $\wt\left(\Theta_{w}(x)\right)=-w\wt\left(x\right)$. \end{defn}
\begin{rem}
Our definition of the quantum twist map $\Theta_{w}$ seems different
from the one in \cite[Section 6.1]{lenagan2015prime}. However these
definitions are the same because they adopt the different antipode from ours. \end{rem}
\begin{prop}
\label{p:rootvector} For $w\in W$ and $\bm{i}=\left(i_{1},\cdots,i_{\ell}\right)\in I\left(w\right)$,
we have
\begin{align*}
\Theta_{w^{-1}}\left(T_{i_{1}}\cdots T_{i_{k-1}}\left(f_{i_{k}}\right)\right) & =T_{i_{\ell}}\cdots T_{i_{k+1}}\left(f_{i_{k}}\right)\ \text{for}\ k=1,\dots,\ell.
\end{align*}
\end{prop}
\begin{proof}
It can be easily checked that
\[
\left(T_{i}\circ S\circ\vee\right)\left(f_{i}\right)=f_{i}.
\]
Hence by Lemma \ref{l:T-antipode} we have
\begin{align*}
\Theta_{w^{-1}}\left(T_{i_{1}}\cdots T_{i_{k-1}}\left(f_{i_{k}}\right)\right) & =\left(T_{i_{\ell}}\cdots T_{i_{1}}\circ S\circ\vee\right)\left(T_{i_{1}}\cdots T_{i_{k-1}}\left(f_{i_{k}}\right)\right)\\
 & =\left(T_{i_{\ell}}\cdots T_{i_{k}}\circ S\circ\vee\right)\left(f_{i_{k}}\right)\\
 & =\left(T_{i_{\ell}}\cdots T_{i_{k+1}}\right)\left(f_{i_{k}}\right).
\end{align*}
 \end{proof}
\begin{cor}
For $w\in W$, we have $\mathbf{U}_{q}^{-}\cap\Theta_{w}(\mathbf{U}_{q}^{-})=\Uq^{-}(w)$.
\end{cor}
\begin{proof}
Since $\Theta_{w}(\mathbf{U}_{q}^{-})=\left(T_{w}\circ S\circ\vee\right)\left(\mathbf{U}_{q}^{-}\right)=\left(T_{w}\circ S\right)\left(\mathbf{U}_{q}^{+}\right)\subset T_{w}(\Uq^{+}\Uq^{0})$,
we obtain $\mathbf{U}_{q}^{-}\cap\Theta_{w}(\mathbf{U}_{q}^{-})\subset\mathbf{U}_{q}^{-}\cap T_{w}(\Uq^{+}\Uq^{0})=\Uq^{-}(w)$. On the other hand, by Proposition \ref{p:rootvector},  
we have $T_{i_{1}}\cdots T_{i_{k-1}}\left(f_{i_{k}}\right)\in\mathbf{U}_{q}^{-}\cap\Theta_{w}(\mathbf{U}_{q}^{-})$, where $\bm{i}=\left(i_{1},\cdots,i_{\ell}\right)\in I\left(w\right)$.
Because $\Uq^{-}(w)$ is generated
by $\left\{ T_{i_{1}}\cdots T_{i_{k-1}}\left(f_{i_{k}}\right)\right\} _{1\leq k\leq\ell}$,
we obtain $\Uq^{-}(w)\subset\mathbf{U}_{q}^{-}\cap\Theta_{w}(\mathbf{U}_{q}^{-})$.
 \end{proof}
\begin{defn}
(1) For $\bm{i}=\left(i_{1},\cdots,i_{\ell}\right)\in I\left(w\right)$,
we set $\bm{i}^{{\rm rev}}=\left(i_{\ell},\dots,i_{1}\right)\in I\left(w^{-1}\right)$.

(2) For $\bm{c}=\left(c_{1},\dots,c_{\ell}\right)\in\mathbb{Z}_{\geq0}^{\ell}$,
we set $\bm{c}^{\mathrm{rev}}:=\left(c_{\ell},\dots,c_{1}\right)\in\mathbb{Z}_{\geq0}^{\ell}$. \end{defn}
\begin{prop}
\label{p:PBW} For $w\in W$, $\bm{i}\in I\left(w\right)$
and $\bm{c}\in\mathbb{Z}_{\geq0}^{\ell}$, we have
\begin{align*}
\Theta_{w^{-1}}\left(F^{\mathrm{up}}\left(\bm{c},\bm{i}\right)\right) & =F^{\mathrm{up}}\left(\bm{c}^{\mathrm{rev}},\bm{i}^{\mathrm{rev}}\right).
\end{align*}
\end{prop}
\begin{proof}
By the equality \eqref{norm}, we have
\[
\left(F^{\mathrm{low}}\left(\bm{c},\bm{i}\right),F^{\mathrm{low}}\left(\bm{c},\bm{i}\right)\right)_{L}=\left(F^{\mathrm{low}}\left(\bm{c}^{\mathrm{rev}},\bm{i}^{\mathrm{rev}}\right),F^{\mathrm{low}}\left(\bm{c}^{\mathrm{rev}},\bm{i}^{\mathrm{rev}}\right)\right)_{L}.
\]
Hence it suffices to show that $\Theta_{w^{-1}}\left(F^{\mathrm{low}}\left(\bm{c},\bm{i}\right)\right)=F^{\mathrm{low}}\left(\bm{c}^{\mathrm{rev}},\bm{i}^{\mathrm{rev}}\right)$.
This follows immediately from Proposition \ref{p:rootvector}.
 \end{proof}
By Proposition \ref{p:PBW}, $\Theta_{w^{-1}}$ is also regarded as
a $\mathbb{Q}(q)$-algebra anti-isomorphism from $\Uq^{-}(w)$ to
$\Uq^{-}(w^{-1})$.
\begin{prop}
\label{p:twist-and-dualbar} Let $w\in W$. For $x\in\Uq^{-}(w)$, we have
\[
\left(\Theta_{w^{-1}}\circ\sigma\right)\left(x\right)=\left(\sigma\circ\Theta_{w^{-1}}\right)\left(x\right).
\]
\end{prop}
\begin{proof}
We may assume that $x$ is homogeneous. On generators, by Proposition  \ref{p:rootdual},
we have
\begin{align*}
\left(\Theta_{w^{-1}}\circ\sigma\right)\left(F^{{\rm up}}\left(\bm{c}_{k},\bm{i}\right)\right) & =\Theta_{w^{-1}}\left(F^{{\rm up}}\left(\bm{c}_{k},\bm{i}\right)\right)\\
 & =F^{{\rm up}}\left(\bm{c}_{k}^{\mathrm{rev}},\bm{i}^{\mathrm{rev}}\right)=\sigma\left(F^{{\rm up}}\left(\bm{c}_{k}^{\mathrm{rev}},\bm{i}^{\mathrm{rev}}\right)\right)\\
 & =\left(\sigma\circ\Theta_{w^{-1}}\right)\left(F^{{\rm up}}\left(\bm{c}_{k},\bm{i}\right)\right).
\end{align*}
Assume that the desired equality holds for homogeneous elements $x',x''\in\Uq^{-}(w)$.
Then, by Proposition \ref{p:dualbar}, we have
\begin{align*}
\left(\Theta_{w^{-1}}\circ\sigma\right)\left(x'x''\right) & =q^{\left(\wt\left(x'\right),\wt\left(x''\right)\right)}\Theta_{w^{-1}}\left(\sigma\left(x''\right)\sigma\left(x'\right)\right)\\
 & =q^{\left(\wt\left(x'\right),\wt\left(x''\right)\right)}\Theta_{w^{-1}}\left(\sigma\left(x'\right)\right)\Theta_{w^{-1}}\left(\sigma\left(x''\right)\right)\\
 & =q^{\left(-w^{-1}\wt\left(x'\right),-w^{-1}\wt\left(x''\right)\right)}\sigma\left(\Theta_{w^{-1}}\left(x'\right)\right)\sigma\left(\Theta_{w^{-1}}\left(x''\right)\right)\\
 & =\sigma\left(\Theta_{w^{-1}}\left(x''\right)\Theta_{w^{-1}}\left(x'\right)\right)\\
 & =\left(\sigma\circ\Theta_{w^{-1}}\right)\left(x'x''\right).
\end{align*}

Hence we obtained the assertion.  \end{proof}
Now we show that the quantum twist maps induce bijections between dual canonical bases of quantum nilpotent subalgebras. Recall Definition \ref{d:PBWparam}.

\begin{thm}
\label{t:mainthm} Let $w\in W$ and $\bm{i}\in I(w)$. For $G^{{\rm up}}\left(b\left(\bm{c},\bm{i}\right)\right)\in\mathbf{B}^{\mathrm{up}}\cap\Uq^{-}(w)$,
we have
\[
\Theta_{w^{-1}}\left(G^{{\rm up}}\left(b\left(\bm{c},\bm{i}\right)\right)\right)=G^{{\rm up}}\left(b\left(\bm{c}^{\mathrm{rev}},\bm{i}^{\mathrm{rev}}\right)\right)\in\mathbf{B}^{\mathrm{\mathrm{up}}}\cap\Uq^{-}(w^{-1}).
\]
\end{thm}
\begin{proof}
We have already checked that $\Theta_{w^{-1}}\left(G^{{\rm up}}\left(b\left(\bm{c},\bm{i}\right)\right)\right)\in\Uq^{-}(w^{-1})$.
Hence by Remark \ref{r:characterize} we only have to show that
\begin{align*}
\sigma\left(\Theta_{w^{-1}}\left(G^{{\rm up}}\left(b\left(\bm{c},\bm{i}\right)\right)\right)\right) & =\Theta_{w^{-1}}\left(G^{{\rm up}}\left(b\left(\bm{c},\bm{i}\right)\right)\right),\\
\Theta_{w^{-1}}\left(G^{{\rm up}}\left(b\left(\bm{c},\bm{i}\right)\right)\right)-F^{{\rm {up}}}\left(\bm{c}^{\mathrm{rev}},\bm{i}^{\mathrm{rev}}\right) & \in\sum_{\bm{c}'\in\mathbb{Z}_{\geq0}^{\ell}}q\mathbb{Z}[q]F^{{\rm {up}}}\left(\bm{c}',\bm{i}^{\mathrm{rev}}\right).
\end{align*}

The latter follows from Proposition \ref{p:dualcanonical} and Proposition
\ref{p:PBW}. The former follows from Proposition \ref{p:twist-and-dualbar}. \end{proof}
By applying $\Theta_{w^{-1}}$ to the expansion of the dual PBW type basis into the dual canonical basis in Remark \ref{r:dualcanonical}, we obtain the following corollary. This symmetry is new when $\mathfrak{g}$ is not finite dimensional. See also Remark \ref{r:finite}. 
\begin{cor}
\label{c:right-lex} Recall the notation in Remark \ref{r:dualcanonical}. For the expansion of the dual PBW type basis into
the dual canonical basis, we have
\begin{equation}\label{reverse}
\left[F^{\mathrm{up}}\left(\bm{c},\bm{i}\right)\colon G^{{\rm up}}\left(b\left(\bm{c}',\bm{i}\right)\right)\right]=\left[F^{\mathrm{up}}\left(\bm{c}^{\mathrm{rev}},\bm{i}^{\mathrm{rev}}\right)\colon G^{{\rm up}}\left(b\left(\left(\bm{c}'\right)^{\mathrm{rev}},\bm{i}^{\mathrm{rev}}\right)\right)\right].
\end{equation}
In particular, we can write the expansion as follows:
\begin{equation*}
F^{\mathrm{up}}\left(\bm{c},\bm{i}\right)=G^{{\rm up}}\left(b\left(\bm{c},\bm{i}\right)\right)+\sum_{\bm{c}'<\bm{c},\;\bm{c}'<_{\mathrm{r}}\bm{c}}\left[F^{\mathrm{up}}\left(\bm{c},\bm{i}\right)\colon G^{{\rm up}}\left(b\left(\bm{c}',\bm{i}\right)\right)\right]G^{{\rm up}}\left(b\left(\bm{c}',\bm{i}\right)\right),
\end{equation*}
here $<_{\mathrm{r}}$ denotes the right lexicographic order on $\mathbb{Z}_{\geq 0}^{\ell}$, which is determined by the condition that $\bm{c}'<_{\mathrm{r}}\bm{c}$ if and only if $\left(\bm{c}'\right)^{\mathrm{rev}}<\bm{c}^{\mathrm{rev}}$.
\end{cor}
\begin{rem}\label{r:NO}
Denote by $\Phi_{+}$ the set of the positive roots of $\mathfrak{g}$. For $\beta\in \Phi_+$, let $\mathfrak{g}_{-\beta}$ be the root space of $-\beta$ and fix a root vector $F_{-\beta}\in \mathfrak{g}_{-\beta}$. Take $w\in W$ and $\bm{i}=(i_1,\dots, i_{\ell})\in I(w)$. Set $\mathfrak{n}^-(w):=\bigoplus_{\beta\in \Phi_+\cap -w\Phi_+}\mathfrak{g}_{-\beta}$. This is a Lie subalgebra of $\mathfrak{g}$. 
Let ${}^{\mathcal{A}}\Uq^-(w):=\sum_{\bm{c}\in\mathbb{Z}_{\geq 0}^{\ell}}\mathcal{A}\Gup\left(b(\bm{c},\bm{i})\right)$. Regard $\mathbb{C}$ as an $\mathcal{A}$-module via $q\mapsto 1$. Then $\mathbb{C}\otimes_{\mathcal{A}}({{}^{\mathcal{A}}\Uq^-(w)})$ is isomorphic to the graded dual $\mathbf{U}(\mathfrak{n}^-(w))^{\ast}_{\mathrm{gr}}$ of the universal enveloping algebra of $\mathfrak{n}^-(w)$ with respect to the usual $\sum_{i\in I}\mathbb{Z}_{\leq 0}\alpha_i$-grading. Consider the map $\mathcal{F}_{\bm{i}}\colon \mathbf{U}(\mathfrak{n}^-(w))^{\ast}_{\mathrm{gr}}\to \mathbb{C}[t_1,\dots, t_{\ell}]$ given by 
\[
f\mapsto \sum_{a_1,\dots, a_{\ell}\in \mathbb{Z}_{\geq 0}}\langle f, F_{-\beta_1}^{a_1}\cdots F_{-\beta_{\ell}}^{a_{\ell}}\rangle\frac{t_1^{a_1}\cdots t_{\ell}^{a_{\ell}}}{a_1!\cdots a_{\ell}!},
\]
here $\Phi_+\cap -w\Phi_+=\{\beta_{k}:=s_{i_1}\cdots s_{i_{k-1}}\alpha_{i_k}\mid k=1,\dots, \ell\}$. Then the map $\mathcal{F}_{\bm{i}}$ is an injective well-defined $\mathbb{C}$-algebra homomorphism. 

By the way, the lexicographic orders on $\mathbb{Z}_{\geq 0}^{\ell}$ naturally induce the total orders on the monomials in $\mathbb{C}[t_1,\dots, t_{\ell}]$. Thus Corollary \ref{c:right-lex} says that the highest term of $\mathcal{F}_{\bm{i}}(1\otimes\Gup(b(\bm{c},\bm{i})))$ with respect to the left lexicographic order is equal to its highest term with respect to the right lexicographic order (and it corresponds to $\bm{c}$). Compare this fact with the results in \cite{FO}.
\end{rem}
\subsection{Quantum twist maps and unipotent quantum minors}\label{Quantum twist maps and unipotent quantum minors}
Unipotent quantum minors are typical and manageable elements of dual canonical bases. In this subsection, we show that the images of certain unipotent quantum minors under quantum twist maps are also described by unipotent quantum minors (Theorem \ref{t:mainthm2}). 

Set $P_{+}:=\{\lambda\in P\mid\langle\lambda,h_{i}\rangle\geq0\;\text{for all}\;i\in I\}$.
For $\lambda\in P_{+}$, denote by $V(\lambda)$ (resp.~$V(-\lambda)$)
the integrable highest (resp.~lowest) weight $\Uq$-module generated
by a highest (resp.~lowest) weight vector $v_{\lambda}$ (resp.~$v_{-\lambda}$)
of weight $\lambda$ (resp.~$-\lambda$).
\begin{defn}
For $i\in I$, denote by $T_{i}$ the $\mathbb{Q}(q)$-linear isomorphism $T''_{i,1}\colon V(\pm\lambda)\to V(\pm\lambda)$ in \cite[Chapter 5]{Lus:intro}.
\end{defn}
The important properties of $T_{i}$ in this paper are the following
\cite[Chapter 37, 39]{Lus:intro}:
\begin{prop}
\textup{(1)} For $x\in\Uq$ and $v\in V(\pm\lambda)$, we have $T_{i}(x.v)=T_{i}(x).T_{i}(v)$.

\textup{(2)} For $w\in W$, the composition map $T_{w}:=T_{i_{1}}\cdots T_{i_{\ell}}\colon V(\pm\lambda)\to V(\pm\lambda)$
does not depend on the choice of $(i_{1},\dots,i_{\ell})\in I(w)$. \end{prop}
\begin{defn}[{{{\cite[Lemma 39.1.2]{Lus:intro}}}}]
For $\lambda\in P_{+}$ and $w\in W$, define the elements $v_{\pm w\lambda}\in V(\pm\lambda)$
by
\[
v_{\pm w\lambda}:=\left(T_{w^{\mp1}}\right)^{\mp1}(v_{\pm\lambda}).
\]
These vectors $v_{\pm w\lambda}$ are called the extremal weight vector
of weight $\pm w\lambda$. In fact, we have
\[
\begin{array}{l}
v_{w\lambda}=f_{i_{1}}^{(\langle s_{i_{2}}\cdots s_{i_{\ell}}\lambda,h_{i_{1}}\rangle)}\cdots f_{i_{\ell-1}}^{(\langle s_{i_{\ell}}\lambda,h_{i_{\ell-1}}\rangle)}f_{i_{\ell}}^{(\langle\lambda,h_{i_{\ell}}\rangle)}.v_{\lambda}\\
v_{-w\lambda}=e_{i_{1}}^{(\langle s_{i_{2}}\cdots s_{i_{\ell}}\lambda,h_{i_{1}}\rangle)}\cdots e_{i_{\ell-1}}^{(\langle s_{i_{\ell}}\lambda,h_{i_{\ell-1}}\rangle)}e_{i_{\ell}}^{(\langle\lambda,h_{i_{\ell}}\rangle)}.v_{-\lambda}
\end{array}
\]
for $(i_{1},\dots,i_{\ell})\in I(w)$.
\end{defn}
The following proposition is well-known and easily checked.
\begin{prop}
\textup{(1)} For $\lambda\in P_{+}$, there exists a unique $\mathbb{Q}(q)$-bilinear
form $(\;,\;)_{\pm\lambda}\colon V(\pm\lambda)\times V(\pm\lambda)\to\mathbb{Q}(q)$
such that
\begin{align*}
 &  & \left(v_{\pm\lambda},v_{\pm\lambda}\right)_{\pm\lambda} & =1 & (x.v_{1},v_{2})_{\pm\lambda} & =(v_{1},\varphi(x).v_{2})_{\pm\lambda}
\end{align*}
$\text{for}\;v_{1},v_{2}\in V(\pm\lambda)\;\text{and}\;x\in\Uq$.
Moreover the form $(\;,\;)_{\pm\lambda}$ is non-degenerate and symmetric.

\textup{(2)} We have $(v_{\pm w\lambda},v_{\pm w\lambda})_{\pm\lambda}=1$
for all $w\in W$.
\end{prop}
\begin{defn}
For $\lambda\in P_{+}$ and $u,w\in W$ with $\pm(u\lambda-w\lambda)\in-\sum_{i\in I}\mathbb{Z}_{\geq0}\alpha_{i}$,
define the element $D_{\pm u\lambda,\pm w\lambda}\in\Uq^{-}$ by the
following property:
\[
(D_{\pm u\lambda,\pm w\lambda},x)_{L}=(v_{\pm u\lambda},x.v_{\pm w\lambda})_{\pm\lambda}
\]
for all $x\in\Uq^{-}$. The element $D_{\pm u\lambda,\pm w\lambda}$ is called a unipotent quantum minor and we note that $\wt\left(D_{\pm u\lambda,\pm w\lambda}\right)=\pm(u\lambda-w\lambda)$.
See \cite[Section 6]{MR2914878}.
\end{defn}
The unipotent quantum minors associated with lowest weight modules
are related with those associated with highest weight modules via
$*$-involution.
\begin{prop}
\label{p:lowtohi} For $\lambda\in P_{+}$ and $u,w\in W$ with $-u\lambda+w\lambda\in-\sum_{i\in I}\mathbb{Z}_{\geq0}\alpha_{i}$,
we have
\[
\ast D_{-u\lambda,-w\lambda}=D_{w\lambda,u\lambda}.
\]
\end{prop}
\begin{proof}
For all $x\in\Uq^{-}$, we have
\begin{align*}
\left(\ast D_{-u\lambda,-w\lambda},x\right)_{L} & =\left(D_{-u\lambda,-w\lambda},\ast x\right)_{L}\\
 & =(v_{-u\lambda},\ast(x).v_{-w\lambda})_{-\lambda}=(v_{-w\lambda},(x)^{\vee}.v_{-u\lambda})_{-\lambda}.
\end{align*}
We can consider the new $\Uq$-module $V(-\lambda)^{\vee}$ which
has the same underlying vector space as $V(-\lambda)$ and is endowed
with the action $\bullet$ of $\Uq$ given by $x\bullet v=(x)^{\vee}.v$
for $x\in\Uq$ and $v\in V(-\lambda)^{\vee}$. Then there exists the
$\Uq$-module isomorphism $\Phi:V(\lambda)\to V(-\lambda)^{\vee}$
given by $v_{\lambda}\mapsto v_{-\lambda}$. Moreover $\Phi(v_{w'\lambda})=v_{-w'\lambda}$
for all $w'\in W$. Indeed, for $(i_{1},\dots,i_{\ell})\in I(w')$,
\begin{align*}
\Phi(v_{w'\lambda}) & =\Phi\left(f_{i_{1}}^{(\langle s_{i_{2}}\cdots s_{i_{\ell}}\lambda,h_{i_{1}}\rangle)}\cdots f_{i_{\ell-1}}^{(\langle s_{i_{\ell}}\lambda,h_{i_{\ell-1}}\rangle)}f_{i_{\ell}}^{(\langle\lambda,h_{i_{\ell}}\rangle)}.v_{\lambda}\right)\\
 & =e_{i_{1}}^{(\langle s_{i_{2}}\cdots s_{i_{\ell}}\lambda,h_{i_{1}}\rangle)}\cdots e_{i_{\ell-1}}^{(\langle s_{i_{\ell}}\lambda,h_{i_{\ell-1}}\rangle)}e_{i_{\ell}}^{(\langle\lambda,h_{i_{\ell}}\rangle)}.v_{-\lambda}=v_{-w'\lambda}.
\end{align*}
Hence
\begin{align*}
\left(\ast D_{-u\lambda,-w\lambda},x\right)_{L} & =(v_{-w\lambda},(x)^{\vee}.v_{-u\lambda})_{-\lambda}\\
 & =(v_{-w\lambda},(\Phi\circ\Phi^{-1})(x\bullet v_{-u\lambda}))_{-\lambda}\\
 & =\left(v_{-w\lambda},\Phi(x.v_{u\lambda})\right)_{-\lambda}\\
 & =(v_{w\lambda},x.v_{u\lambda})_{-\lambda}=\left(D_{w\lambda,u\lambda},x\right)_{L}
\end{align*}
for all $x\in\Uq^{-}$. This proves the proposition. \end{proof}
\begin{prop}[{{\cite[Proposition 4.1]{MR1240605}}}]
\label{p:minor-dualcanonical} The unipotent quantum minors are elements
of $\mathbf{B}^{\mathrm{up}}$.
\end{prop}

We consider the unipotent quantum minors which belong to $\Uq^{-}(w)$.
\begin{prop}[{{{\cite[Proposition 3.4]{kimura2015remarks}}}}]
\label{p:intersection} For $w\in W$ and $\bm{i}=(i_{1},\dots,i_{\ell})\in I(w)$,
we have
\[
\Uq^{-}\cap T_{w}\left(\Uq^{-}\right)=\Uq^{-}\cap T_{i_{1}}\left(\Uq^{-}\right)\cap T_{i_{1}}T_{i_{2}}\left(\Uq^{-}\right)\cap\cdots\cap T_{i_{1}}T_{i_{2}}\cdots T_{i_{\ell}}\left(\Uq^{-}\right).
\]
\end{prop}
\begin{lem}
\label{l:perp} For $w\in W$, set $\Uq^{-}(w)^{\perp}:=\{x\in\Uq^{-}\mid(x,\Uq^{-}(w))_{L}=0\}$.
Then
\[
\Uq^{-}(w)^{\perp}=\Uq^{-}(w)(\Uq^{-}\cap T_{w}(\Uq^{-})\cap\Ker\varepsilon).
\]
\end{lem}
\begin{proof}
By \cite[Theorem 1.1]{kimura2015remarks}, we have a decomposition
$\Uq^{-}=\Uq^{-}(w)\oplus\Uq^{-}(w)(\Uq^{-}\cap T_{w}(\Uq^{-})\cap\Ker\varepsilon)$
as a $\mathbb{Q}(q)$-vector space. By the way, we also have $\Uq^{-}=\Uq^{-}(w)\oplus\Uq^{-}(w)^{\perp}$.

Hence it suffices for us to prove the following inclusion:
\[
\Uq^{-}(w)(\Uq^{-}\cap T_{w}(\Uq^{-})\cap\Ker\varepsilon)\subset\Uq^{-}(w)^{\perp}.
\]
It is shown by using Proposition \ref{p:compatibility} and Proposition
\ref{p:intersection} repeatedly.  \end{proof}
\begin{prop}
\label{p:qminor} Let $\lambda\in P_{+}$ and $u_{1},u_{2},w\in W$.
Suppose that $u_{2}$ is less than or equal to $w$ with respect to
the weak right Bruhat order, that is, $\ell (w)=\ell (u_2)+\ell(u_2^{-1}w)$. Here $\ell(u)$ denotes the length of $u$ for $u\in W$. 
Then
\[
D_{-u_{1}\lambda,-u_{2}\lambda}\in\Uq^{-}(w).
\]
\end{prop}
\begin{proof}
It suffices to show that
\[
(v_{-u_{1}\lambda},\Uq^{-}(w)^{\perp}.v_{-u_{2}\lambda})_{-\lambda}=0.
\]
For every homogeneous element $x\in\Uq^{-}\cap T_{w}(\Uq^{-})\cap\Ker\varepsilon$,
we have $u_{2}^{-1}\wt x\in\sum_{i\in I}\mathbb{Z}_{\leq0}\alpha_{i}$
by Proposition \ref{p:intersection}. Here note that there exists $\bm{i}=(i_{1},i_{2},\dots,i_{\ell})\in I(w)$ such that $(i_{1},i_{2},\dots,i_{k})\in I(u_{2})$ for some $k\in\{1,\dots,\ell\}$. Therefore,
\[
x.v_{-u_{2}\lambda}=T_{u_{2}}\left((T_{u_{2}})^{-1}(x).v_{-\lambda}\right)=0.
\]
Hence Lemma \ref{l:perp} implies the assertion.  \end{proof}
\begin{lem}
\label{l:form-preserve} Let $w\in W$. Then we have
\[
\left(x,y\right)_{L}=\left(\Theta_{w^{-1}}(x),\Theta_{w^{-1}}(y)\right)_{L}\;\text{for all}\;x,y\in\Uq^{-}(w).
\]
\end{lem}
\begin{proof}
This follows immediately from Proposition \ref{p:PBW}.  \end{proof}
The following theorem states that the quantum twist maps induce bijections between the unipotent quantum minors satisfying certain conditions on elements of the Weyl group. This is a quantum analogue of \cite[Lemma 2.25]{MR1652878}. 
\begin{thm}
\label{t:mainthm2} Let $\lambda\in P_{+}$ and $u_{1},u_{2}\in W$.
Suppose that $u_{1}$ and $u_{2}$ are less than or equal to $w$
with respect to the weak right Bruhat order. Then we have
\[
\Theta_{w^{-1}}(D_{-u_{1}\lambda,-u_{2}\lambda})=D_{-w^{-1}u_{2}\lambda,-w^{-1}u_{1}\lambda}.
\]
\end{thm}
\begin{proof}
By Proposition \ref{p:qminor}, we have $D_{-u_{1}\lambda,-u_{2}\lambda}\in\Uq^{-}(w)$.
Therefore we have $\Theta_{w^{-1}}(D_{-u_{1}\lambda,-u_{2}\lambda})\in\Uq^{-}(w^{-1})$.
By Lemma \ref{l:form-preserve}, for $x\in\Uq^{-}(w^{-1})$,
\begin{align*}
(\Theta_{w^{-1}}(D_{-u_{1}\lambda,-u_{2}\lambda}),x)_{L} & =\left(D_{-u_{1}\lambda,-u_{2}\lambda},\Theta_{w}(x)\right)_{L}\\
 & =\left(v_{-u_{1}\lambda},\Theta_{w}(x).v_{-u_{2}\lambda}\right)_{-\lambda}\\
 & =\left(v_{-u_{2}\lambda},\left(\varphi\circ\Theta_{w}\right)(x).v_{-u_{1}\lambda}\right)_{-\lambda}.
\end{align*}
Now $\varphi\circ\Theta_{w}$ is a $\mathbb{Q}(q)$-algebra automorphism
of $\Uq$. Hence we can consider the new $\Uq$-module $V'(-\lambda)$
which has the same underlying vector space as $V(-\lambda)$ and is
endowed with the action $\star$ of $\Uq$ given by $x\star v=(\varphi\circ\Theta_{w})(x).v$
for $x\in\Uq$ and $v\in V'(-\lambda)$.

Then there exists a $\Uq$-module isomorphism $V(-\lambda)\to V'(-\lambda)$
given by $v_{-\lambda}\mapsto v_{-w\lambda}$. Note that $(\varphi\circ\Theta_{w})(q^{h})=q^{w(h)}$
for $h\in P^{*}$. Hence the vector $v_{-u_{i}\lambda}\in V'(-\lambda)$
is a vector of weight $-w^{-1}u_{i}\lambda$ ($i=1,2$). Moreover
it is well-known that the weight space of $V(-\lambda)$ of weight
$\mu$ is $1$-dimensional for all $\mu\in-W\lambda$. Therefore as
in the proof of Proposition \ref{p:lowtohi} we have
\begin{align*}
\left(v_{-u_{2}\lambda},\left(\varphi\circ\Theta_{w}\right)(x).v_{-u_{1}\lambda}\right)_{-\lambda} & =\zeta\left(v_{-w^{-1}u_{2}\lambda},x.v_{-w^{-1}u_{1}\lambda}\right)_{-\lambda}\\
 & =\zeta\left(D_{-w^{-1}u_{2}\lambda,-w^{-1}u_{1}\lambda},x\right)_{L}
\end{align*}
for some $\zeta\in\mathbb{Q}(q)^{\times}$ and all $x\in\Uq^{-}(w^{-1})$.
By our assumption, $w^{-1}u_{1}$ is less than or equal to $w^{-1}$
with respect to the weak right Bruhat order. Therefore $D_{-w^{-1}u_{2}\lambda,-w^{-1}u_{1}\lambda}\in\Uq^{-}(w^{-1})$
by Proposition \ref{p:qminor}. Hence $\Theta_{w^{-1}}(D_{-u_{1}\lambda,-u_{2}\lambda})=\zeta D_{-w^{-1}u_{2}\lambda,-w^{-1}u_{1}\lambda}$.

On the other hand, by Theorem \ref{t:mainthm}, $\Theta_{w^{-1}}(D_{-u_{1}\lambda,-u_{2}\lambda})\in\mathbf{B}^{\mathrm{up}}\cap\Uq^{-}(w^{-1})$.
Therefore, by Proposition \ref{p:minor-dualcanonical}, $\zeta=1$
and $\Theta_{w^{-1}}(D_{-u_{1}\lambda,-u_{2}\lambda})=D_{-w^{-1}u_{2}\lambda,-w^{-1}u_{1}\lambda}$.
 \end{proof}
As a corollary of Theorem \ref{t:mainthm2}, we show that quantum twist maps preserve a quantum analogue of specific determinantal identities, called a quantum $T$-system.
\begin{defn}
Let $w\in W$ and $\bm{i}=(i_1,\dots, i_{\ell})\in I(w)$. For $0\leq b \leq d\leq \ell$ and $j\in I$, we set 
\[
D\left(b, d; j\right)(=D^{\bm{i}}(b, d; j)):=D_{-\mu(b, j), -\mu(d, j)},
\]
here $\mu(b, j)(=\mu^{\bm{i}}(b, j)):=s_{i_1}\cdots s_{i_{b}}\varpi_j$. 
Moreover, when $i_b=i_d=j$, we write $D(b, d)(=D^{\bm{i}}(b, d)):=D(b, d; j)$. Note that $D(0, d)=D_{-\varpi_{i_d}, -s_{i_1}\cdots s_{i_d}\varpi_{i_d}}$ for $1\leq d\leq \ell$. For $1\leq d\leq \ell$ and $j\in I$, write
\begin{align*}
b^-&:=\max\left(\{0\}\cup \{1\leq b'\leq b-1\mid i_{b'}=i_b\}\right),\\
b^-(j)&:=\max\left(\{0\}\cup \{1\leq b'\leq b-1\mid i_{b'}=j\}\right).
\end{align*}
Then, for $0\leq b \leq d\leq \ell$ and $j\in I$, we have $D\left(b, d; j\right)=D\left(b^-(j), d^-(j)\right)$. 
\end{defn}
\begin{prop}[{\cite[Proposition 5.5]{GLS:qcluster}}]\label{p:T-sys}
Let $w\in W$ and $\bm{i}=(i_1,\dots, i_{\ell})\in I(w)$. Fix an arbitrary total order on $I$. Suppose that the integers $b, d$ satisfy that $1\leq b<d\leq \ell$ and $i_b=i_d=i$. Then we have
\begin{align}
q^{A}D(b^-, d^-)D(b, d)&=q_i^{-1}q^{B}D(b, d^-)D(b^-, d)+q^C\dprod_{j\in I\setminus\{i\}}D(b^-(j), d^-(j))^{-a_{ji}}\label{T-sys1}\\
&=q_i^{-1}q^{B'}D(b^-, d)D(b, d^-)+q^C\dprod_{j\in I\setminus\{i\}}D(b^-(j), d^-(j))^{-a_{ji}},\label{T-sys2}
\end{align}
here 
\begin{align*}
&a_{ji}=\langle h_j, \alpha_i\rangle,&&&
&A=\left(\mu(b, i), \mu(b^-, i)-\mu(d^-, i)\right),\\
B=&\left(\mu(b^-, i), \mu(b, i)-\mu(d^-, i)\right),&&&
&B'=\left(\mu(b, i), \mu(b^-, i)-\mu(d, i)\right),
\end{align*}
\begin{align*}
C=&\sum_{j\in I\setminus\{i\}}\left(\begin{array}{c}
-a_{ji}\\
2
\end{array}\right)\left(\mu(b, j),\mu(b, j)-\mu(d, j)\right)\\
&+\sum_{j, k\in I\setminus\{i\};k<j}a_{ji}a_{ki}\left(\mu(b, j), \mu(b, k)-\mu(d, k)\right),
\end{align*}
and $\dprod$ denotes a product with respect to the increasing order from left to right. This system of equalities is called the quantum $T$-system in $\Uq^-(w)$.
\end{prop}
\begin{rem}
Note that our convention is different from the one in \cite{GLS:qcluster}, and Gei\ss-Leclerc-Schr\"{o}er  always assume that $\mathfrak{g}$ is symmetric. Nevertheless, we can prove the equality above in the same manner as in \cite{GLS:qcluster}. 
\end{rem}
\begin{cor}\label{c:T-sys}
The quantum twist map $\Theta_{w^{-1}}$ maps the quantum $T$-system in $\Uq^-(w)$ to the one in $\Uq^-(w^{-1})$.
\end{cor}
\begin{proof}
Fix $\bm{i}=(i_1,\dots, i_{\ell})\in I(w)$. Let $b, d$ the integers such that $1\leq b<d\leq \ell$ and $i_b=i_d=i$. For $a=1, \dots, \ell$, set $a_{\mathrm{r}}:=\ell-a+1$. For simplicity of notation, we write $a_{\mathrm{r}}^-:=(a_{\mathrm{r}})^-$ and $a_{\mathrm{r}}^-(j):=(a_{\mathrm{r}})^-(j)$ where the right-hand sides are considered with respect to $\bm{i}^\mathrm{rev}$ for $a=1, \dots, \ell$. Note that $w^{-1}\mu^{\bm{i}}(a', j)=\mu^{\bm{i}^{\mathrm{rev}}}(\ell -a', j)$ for $a'=0,\dots, \ell$ and $j\in I$. In particular, for $a=b, d$, we have $w^{-1}\mu^{\bm{i}}(a, i)=\mu^{\bm{i}^\mathrm{rev}}(a_{\mathrm{r}}^-, i), w^{-1}\mu^{\bm{i}}(a^-, i)=\mu^{\bm{i}^\mathrm{rev}}(a_{\mathrm{r}}, i)$ and $w^{-1}\mu^{\bm{i}}(a, j)=\mu^{\bm{i}^\mathrm{rev}}(a_{\mathrm{r}}, j)$ if $j\neq i$. Hence, by applying $\Theta_{w^{-1}}$ to both sides of \eqref{T-sys1} and using Theorem \ref{t:mainthm2}, we obtain
\begin{align}\label{T-sysrev}
&q^{A}D^{\bm{i}^{\mathrm{rev}}}(d_{\mathrm{r}}^-, b_{\mathrm{r}}^-)D^{\bm{i}^{\mathrm{rev}}}(d_{\mathrm{r}}, b_{\mathrm{r}})\\
&=q_i^{-1}q^{B}D^{\bm{i}^{\mathrm{rev}}}(d_{\mathrm{r}}^-, b_{\mathrm{r}})D^{\bm{i}^{\mathrm{rev}}}(d_{\mathrm{r}}, b_{\mathrm{r}}^-)+q^C\dprod_{j\in I^{\mathrm{rev}}\setminus\{i\}}D^{\bm{i}^{\mathrm{rev}}}(d_{\mathrm{r}}^-(j), b_{\mathrm{r}}^-(j))^{-a_{ji}},\notag
\end{align}
here $a_{ji}, A, B, C$ are the same as in Proposition \ref{p:T-sys} and $I^{\mathrm{rev}}$ denotes the index set $I$ with the reverse total order. By the way, 
\begin{align*}
&\left(\mu^{\bm{i}}(b, i), \mu^{\bm{i}}(b^-, i)\right)=\left(s_i\varpi_i, \varpi_i\right)=\left(\mu^{\bm{i}}(d, i), \mu^{\bm{i}}(d^-, i)\right),\\
&\left(\mu^{\bm{i}}(a', j), \mu^{\bm{i}}(a', k)\right)=\left(\varpi_j, \varpi_k\right)\;\text{for all}\ a'=0,\dots, \ell\;\text{and}\;j, k\in I.
\end{align*}
Therefore we have
\begin{align*}
A=\left(\mu^{\bm{i}}(b, i), \mu^{\bm{i}}(b^-, i)-\mu^{\bm{i}}(d^-, i)\right)&=\left(\mu^{\bm{i}}(d^-, i), \mu^{\bm{i}}(d, i)-\mu^{\bm{i}}(b, i)\right)\\
&=\left(\mu^{\bm{i}^{\mathrm{rev}}}(d_{\mathrm{r}}, i), \mu^{\bm{i}^{\mathrm{rev}}}(d_{\mathrm{r}}^-, i)-\mu^{\bm{i}^{\mathrm{rev}}}(b_{\mathrm{r}}^-, i)\right),\\
B=\left(\mu^{\bm{i}}(b^-, i), \mu^{\bm{i}}(b, i)-\mu^{\bm{i}}(d^-, i)\right)&=\left(\mu^{\bm{i}}(d^-, i), \mu^{\bm{i}}(d, i)-\mu^{\bm{i}}(b^-, i)\right)\\
&=\left(\mu^{\bm{i}^{\mathrm{rev}}}(d_{\mathrm{r}}, i), \mu^{\bm{i}^{\mathrm{rev}}}(d_{\mathrm{r}}^-, i)-\mu^{\bm{i}}(b_{\mathrm{r}}, i)\right),
\end{align*}
\begin{align*}
C=&\sum_{j\in I\setminus\{i\}}\left(\begin{array}{c}
-a_{ji}\\
2
\end{array}\right)\left(\mu^{\bm{i}}(b, j),\mu^{\bm{i}}(b, j)-\mu^{\bm{i}}(d, j)\right)\\
&+\sum_{j, k\in I\setminus\{i\};k<j}a_{ji}a_{ki}\left(\mu^{\bm{i}}(b, j), \mu^{\bm{i}}(b, k)-\mu^{\bm{i}}(d, k)\right)\\
=&\sum_{j\in I^{\mathrm{rev}}\setminus\{i\}}\left(\begin{array}{c}
-a_{ji}\\
2
\end{array}\right)\left(\mu^{\bm{i}^{\mathrm{rev}}}(d_{\mathrm{r}}, j),\mu^{\bm{i}^{\mathrm{rev}}}(d_{\mathrm{r}}, j)-\mu^{\bm{i}^{\mathrm{rev}}}(b_{\mathrm{r}}, j)\right)\\
&+\sum_{j, k\in I^{\mathrm{rev}}\setminus\{i\};j<k}a_{ki}a_{ji}\left(\mu^{\bm{i}^{\mathrm{rev}}}(d_{\mathrm{r}}, k), \mu^{\bm{i}^{\mathrm{rev}}}(d_{\mathrm{r}}, j)-\mu^{\bm{i}^{\mathrm{rev}}}(b_{\mathrm{r}}, j)\right).
\end{align*}
Therefore the equality \eqref{T-sysrev} belongs to the quantum $T$-system in $\Uq(w^{-1})$.
 \end{proof}
\subsection{Quantum twist maps and cofinite quantum nilpotent subalgebras}\label{Quantum twist maps and cofinite quantum nilpotent subalgebras}
\begin{prop}[{{{\cite[Proposition 3.4.7, Corollary 3.4.8]{MR1265471}, \cite[Theorem 1.2]{MR1409422}
and \cite[Theorem 4.23]{MR2914878}}}}]
\label{p:Saito_refl} Let $i\in I$ and $b\in \mathscr{B}(\infty)$ with $\varepsilon_{i}(b)=0$.
Then we have
\[
T_{i}^{-1}\left(\Gup(b)\right)=\Gup\left(\sigma_{i}^{\ast}b\right),
\]
where $\sigma_{i}^{\ast}\colon\{b\in \mathscr{B}(\infty)\mid\varepsilon_{i}(b)=0\}\to\{b\in \mathscr{B}(\infty)\mid\varepsilon_{i}^{\ast}(b)=0\}$
is the bijection given by $b\mapsto\tilde{f}_{i}^{\varphi_{i}^{\ast}(b)}(\tilde{e}_{i}^{\ast})^{\varepsilon_{i}^{\ast}(b)}b$.
\end{prop}

\begin{prop}[{\cite[Theorem 3.3]{kimura2015remarks}}]
\label{p:cofinite} Let $w\in W$. Then $\Uq^{-}\cap T_{w}(\Uq^{-})\cap\mathbf{B}^{\mathrm{up}}$
is a basis of $\Uq^{-}\cap T_{w}(\Uq^{-})$. \end{prop}
\begin{thm}
\label{t:mainthm3} For $\Gup(b)\in\Uq^{-}\cap T_{w}(\Uq^{-})\cap\mathbf{B}^{\mathrm{up}}$,
we have
\[
\Theta_{w^{-1}}(\Gup(b))=(-1)^{\height \beta_w}q^{\frac{1}{2}(\beta_w,\beta_w)-(\rho,\beta_w)}\Gup(\ast\sigma_{i_{\ell}}^{*}\cdots\sigma_{i_{1}}^{*}b)^{\vee}t_{\beta_w}.
\]
for $(i_{1},\dots,i_{\ell})\in I(w)$. Here we write $\beta_w:=-w^{-1}\wt b$ and set $t_{\alpha}:=\prod_{i\in I}t_{i}^{m_{i}}$
for $\alpha=\sum_{i\in I}m_{i}\alpha_{i}\in Q$. See Proposition \ref{p:dualbar} for the definition of $\height$. \end{thm}
\begin{proof}
Recall that we have $\Theta_{w^{-1}}=S\circ\vee\circ(T_{w})^{-1}$.
By Proposition \ref{p:Saito_refl}, we have $(T_{w})^{-1}\Gup(b)=\Gup(\sigma_{i_{\ell}}^{*}\cdots\sigma_{i_{1}}^{*}b)$
for $(i_{1},\dots,i_{\ell})\in I(w)$ (see also \cite[Proposition 3.4]{kimura2015remarks}).
Moreover, for a homogeneous element $x\in\Uq^{+}$, we have $S(x)=(-1)^{\height\wt x}q^{\frac{1}{2}(\wt x,\wt x)-(\rho,\wt x)}\ast(x)t_{\wt x}$.
Therefore we obtain the assertion.
 \end{proof}

\section{Finite type case}\label{Finite type case}

In this section, we assume that $\mathfrak{g}$ is a finite dimensional
complex simple Lie algebra. Let $w_{0}\in W$ be the longest element
of the corresponding Weyl group $W$. It is well-known that there
is a unique Dynkin diagram automorphism $\theta$ with $-w_{0}\left(\alpha_{i}\right)=\alpha_{\theta\left(i\right)}$
for all $i\in I$. For a reduced word $\bm{i}=\left(i_{1},\cdots,i_{N}\right)\in I\left(w_{0}\right)$,
it is known that $\left(i_{2},\cdots,i_{N},\theta\left(i_{1}\right)\right)$
is also a reduced word of $w_0$.
\begin{defn}
We define a $\mathbb{Q}\left(q\right)$-algebra automorphism on $\Uq\left(\mathfrak{g}\right)$
defined by
\begin{align*}
\theta\left(e_{i}\right) & =e_{\theta\left(i\right)} & \theta\left(f_{i}\right) & =f_{\theta\left(i\right)} & \theta\left(q^{h}\right) & =q^{-w_{0}\left(h\right)}.
\end{align*}
\end{defn}
\begin{prop}[{{{\cite[Proposition 8.20]{MR1359532}, \cite[Proposition 3.2]{nakajima:CBMS1}}}}]
\label{p:simple}
If $w\left(\alpha_{i}\right)=\alpha_{j}$ for some $i, j\in I$ and $w\in W$, we have
\[
T_{w}\left(x_{i}\right)=x_{j}
\]
where $x=e, f$.
\end{prop}

\begin{prop}
\label{p:thetaast} We have $\theta\circ*=\Theta_{w_0}.$ 
\end{prop}
\begin{proof}
By Proposition \ref{p:simple}, we have 
\begin{align*}
T_{w_0}\left( e_{i}\right)=-f_{\theta (i)}t_{\theta (i)}, &  & T_{w_0}\left( f_{i}\right)=-t_{\theta (i)}^{-1}e_{\theta (i)}, &  & T_{w_0}\left(q^{h}\right) & =q^{w_0\left( h\right)}.
\end{align*}
Hence the proposition follows from the straightforward check on the generators of $\Uq$.
 \end{proof}
By Proposition \ref{p:thetaast}, we obtain the following corollary.
\begin{cor}\label{c:Lusztig1}
For $\bm{i}=\left(i_{1},\cdots,i_{N}\right)\in I\left(w_{0}\right)$
and $\bm{c}=\left(c_{1},\cdots,c_{N}\right)\in\mathbb{Z}_{\geq0}^{N}$,
we have
\[
\left(\theta\circ*\right)\left(G^{\mathrm{up}}\left(b\left(\bm{c},\bm{i}\right)\right)\right)=G^{\mathrm{up}}\left(b\left(\bm{c}^{\mathrm{rev}},\bm{i}^{\mathrm{rev}}\right)\right).
\]

\end{cor}
Moreover by Lemma \ref{l:form-preserve} we obtain the following corollary.
\begin{cor}\label{c:Lusztig2}
Let $\bm{i}=\left(i_{1},\cdots,i_{N}\right)\in I\left(w_{0}\right)$
and $\bm{c}=\left(c_{1},\cdots,c_{N}\right)\in\mathbb{Z}_{\geq0}^{N}$.

\textup{(1)} We have
\[
\left(\theta\circ*\right)\left(G^{\mathrm{low}}\left(b\left(\bm{c},\bm{i}\right)\right)\right)=G^{\mathrm{low}}\left(b\left(\bm{c}^{\mathrm{rev}},\bm{i}^{\mathrm{rev}}\right)\right).
\]

\textup{\textup{(2)}} Write $G^{\mathrm{low}}\left(b\left(\bm{c},\bm{i}\right)\right)=\sum_{\bm{c}'}\left(G^{\mathrm{low}}\left(b\left(\bm{c},\bm{i}\right)\right)\colon F^{\mathrm{low}}\left(\bm{c}',\bm{i}\right)\right)F^{\mathrm{low}}\left(\bm{c}',\bm{i}\right)$. Then we have
\[
\left(G^{\mathrm{low}}\left(b\left(\bm{c},\bm{i}\right)\right)\colon F^{\mathrm{low}}\left(\bm{c}',\bm{i}\right)\right)=\left(G^{\mathrm{low}}\left(b\left(\bm{c}^{\mathrm{rev}},\bm{i}^{\mathrm{rev}}\right)\right)\colon F^{\mathrm{low}}\left(\left(\bm{c}'\right)^{\mathrm{rev}},\bm{i}^{\mathrm{rev}}\right)\right).
\]

In particular, we have
\[
G^{\mathrm{low}}\left(b\left(\bm{c},\bm{i}\right)\right)=F^{\mathrm{low}}\left(\bm{c},\bm{i}\right)+\sum_{\substack{\bm{c}<\bm{c}', \bm{c}<_{\mathrm{r}}\bm{c}'}}\left(G^{\mathrm{low}}\left(b\left(\bm{c},\bm{i}\right)\right)\colon F^{\mathrm{low}}\left(\bm{c}',\bm{i}\right)\right)F^{\mathrm{low}}\left(\bm{c}',\bm{i}\right).
\]

\end{cor}
\begin{rem}\label{r:finite}
We have to remark that Corollary \ref{c:Lusztig1} was already proved by Lusztig \cite[2.11]{MR1035415}.  
Moreover, when $\mathfrak{g}$ is of finite type, we can also show the equality \eqref{reverse} in Corollary \ref{c:right-lex} without using quantum twist maps, by the results in \cite[2.11]{MR1035415} together with Proposition \ref{p:Saito_refl}. Note that $\Theta_w=(T_{w_0w^{-1}})^{-1}\circ \theta\circ*$ for all $w\in W$.
\end{rem}




\begin{thebibliography}{10}

\bibliographystyle{amsplain}

\bibitem{MR1712630}
Jonathan Beck, Vyjayanthi Chari, and Andrew Pressley: An algebraic
  characterization of the affine canonical basis. Duke Math. J. \textbf{99}, no.~3, 455--487 (1999).

\bibitem{MR1405449}
Arkady Berenstein, Sergey Fomin, and Andrei Zelevinsky: Parametrizations
  of canonical bases and totally positive matrices. Adv. Math. \textbf{122}, no.~1, 49--149 (1996).

\bibitem{MR3397447}
Arkady Berenstein and Dylan Rupel: Quantum cluster characters of {H}all
  algebras. Selecta Math. (N.S.) \textbf{21}, no.~4, 1121--1176 (2015).

\bibitem{MR1456321}
Arkady Berenstein and Andrei Zelevinsky: Total positivity in {S}chubert
  varieties. Comment. Math. Helv. \textbf{72}, no.~1, 128--166 (1997).

\bibitem{MR1652878}
Sergey Fomin and Andrei Zelevinsky: Double {B}ruhat cells and total
  positivity. J. Amer. Math. Soc. \textbf{12}, no.~2, 335--380 (1999).

\bibitem{FO}
Naoki Fujita and Hironori Oya: A comparison of Newton-Okounkov polytopes of Schubert varieties. J.~Lond.~Math.~Soc.~(2) \textbf{96}, no.~1, 201--227 (2017).


\bibitem{MR2833478}
Christof Geiss, Bernard Leclerc, and Jan Schr{{\"o}}er: Generic bases for
  cluster algebras and the {C}hamber ansatz. J. Amer. Math. Soc. \textbf{25}, no.~1, 21--76 (2012).

\bibitem{GLS:qcluster}
Christof Geiss, Bernard Leclerc, and Jan Schr{{\"o}}er: Cluster structures on quantum coordinate rings. Sel. Math. New Ser., 337--397 (2013).

\bibitem{goodearlyakimov}
Ken Goodearl and Milen Yakimov: Quantum cluster algebra structures on quantum nilpotent algebras. Mem.~Amer.~Math.~Soc. \textbf{247}, no.~1169 (2017).

\bibitem{goodearl2016berenstein}
Ken Goodearl and Milen Yakimov: The {B}erenstein-{Z}elevinsky quantum
  cluster algebra conjecture. arXiv preprint arXiv:1602.00498, 2016.

\bibitem{MR1359532}
Jens~Carsten Jantzen: Lectures on quantum groups. Grad. Stud. Math., vol.~6 (1996).


\bibitem{MR1115118}
Masaki Kashiwara: On crystal bases of the {$Q$}-analogue of universal
  enveloping algebras. Duke Math. J. \textbf{63}, no.~2, 465--516 (1991).

\bibitem{MR1240605}
Masaki Kashiwara: The crystal base and {L}ittelmann's refined {D}emazure character
  formula. Duke Math. J. \textbf{71}, no.~3, 839--858 (1993). 

\bibitem{MR1357199}
Masaki Kashiwara: On crystal bases. Representations of groups ({B}anff, {AB}, 1994), CMS Conf. Proc., vol.~16, Amer. Math. Soc., Providence, RI,  155--197 (1995).

\bibitem{MR2914878}
Yoshiyuki Kimura: Quantum unipotent subgroup and dual canonical basis.  Kyoto J. Math. \textbf{52}, no.~2, 277--331 (2012).

\bibitem{kimura2015remarks}
Yoshiyuki Kimura: Remarks on quantum unipotent subgroups and the dual canonical basis. Pacific J.~Math. \textbf{286}, no.~1, 125--151 (2017). 

\bibitem{lenagan2015prime}
Thomas Lenagan and Milen Yakimov: Prime factors of quantum {S}chubert cell
  algebras and clusters for quantum {R}ichardson varieties. J. reine angew. Math, Ahead of Print. 

\bibitem{MR1035415}
George Lusztig: Canonical bases arising from quantized enveloping
  algebras. J. Amer. Math. Soc. \textbf{3}, no.~2, 447--498 (1990).

\bibitem{MR1409422}
George Lusztig: Braid group action and canonical bases. Adv. Math. \textbf{122}, no.~2, 237--261 (1996).

\bibitem{Lus:intro}
George Lusztig: Introduction to quantum groups. Mod. Birkh{\"a}user Class. New York, Reprint of the 1994 edition (2010).

\bibitem{MR3403455}
Peter~J. McNamara: Finite dimensional representations of
  {K}hovanov-{L}auda-{R}ouquier algebras {I}: {F}inite type. J. Reine Angew.
  Math. \textbf{707}, 103--124 (2015).

\bibitem{nakajima:CBMS1}
Hiraku Nakajima: Quiver varieties and canonical bases of quantum affine
  algebras. available at
  \url{http://www4.ncsu.edu/~jing/conf/CBMS/Nakajima_part1.pdf} (2010).

\bibitem{MR1265471}
Yoshihisa Saito: P{BW} basis of quantized universal enveloping algebras.  Publ. Res. Inst. Math. Sci. \textbf{30}, no.~2, 209--232 (1994).

\bibitem{MR3096792}
Harold Williams: Cluster ensembles and {K}ac-{M}oody groups. Adv. Math.
  \textbf{247}, 1--40 (2013).
\end{thebibliography}
\end{document}